\documentclass{amsart}
\usepackage{amsmath,amsfonts,amssymb}
\usepackage{ifthen}
\usepackage{longtable}
\usepackage{array}
\usepackage{url}

\newcommand{\Z}{\mathbb{Z}}

\newcommand{\Q}{\mathbb{Q}}
\newcommand{\R}{\mathbb{R}}

\newcommand{\ord}{\mathrm{ord}}

\newcommand{\LL}{\mathcal{L}}

\DeclareMathOperator{\Norm}{N}

\newtheorem*{theorem*}{Theorem}
\newtheorem{theorem}{Theorem}
\newtheorem{lemma}{Lemma}
\newtheorem{corollary}{Corollary}
\newtheorem{proposition}{Proposition}

\theoremstyle{remark}
\newtheorem{remark}{Remark}

\begin{document}
\title{Sums of Fibonacci numbers that are perfect powers}

\author{Volker Ziegler}
\address{V. Ziegler,
University of Salzburg,
Hellbrunnerstrasse 34/I,
A-5020 Salzburg, Austria}
\email{volker.ziegler\char'100sbg.ac.at}
\thanks{The author was supported by the Austrian Science Fund(FWF) under the project~I4406.}

\begin{abstract}
Let us denote by $F_n$ the $n$-th Fibonacci number. In this paper we show that for a fixed integer $y$ there exists at most one integer exponent $a>0$ such that the Diophantine equation $F_n+F_m=y^a$ has a solution $(n,m,a)$ in positive integers satisfying $n>m>0$, unless $y=2,3,4,6$ or $10$.
\end{abstract}

\maketitle

\section{Introduction}

Let $F_n$ be the $n$-th Fibonacci number defined by $F_0=0$, $F_1=1$ and $F_{n+2}=F_{n+1}+F_n$ for all $n\geq 0$. Some years ago Bravo and Luca \cite{Bravo:2016} considered the Diophantine equation
$$
F_n + F_m = 2^a,
$$
and showed that the only solutions $(n,m,a)\in \Z^3$ to this equation with $n>m>0$ are
$$(n,m,a)=(2,1,1),(4,1,2),(4,2,2),(5,4,3),(7,4,4).$$
Using the method due to Bravo and Luca \cite{Bravo:2016} we are not restricted to powers of $2$ and we can handle by the same method the Diophantine equation
\begin{equation}\label{eq:main}
F_n+F_m=y^a, \qquad n>m> 0,\;\; a>0
\end{equation}
for any fixed integer $y$. For instance in the case that $y=3$ we find the solutions $(n,m,a)=(3,1,1),(3,2,1),(6,1,2),(6,2,2)$. In this paper we want to show that Diophantine Equation \eqref{eq:main} has in almost all cases at most one solution for fixed integer $y$. That is the cases $y=2$ and $y=3$ are in some sense special cases. Let us note that if there exists a solution $(n,1,a)$ to \eqref{eq:main}, then $(n,2,a)$ is also a solution to \eqref{eq:main} since $F_1=F_2=1$. Thus speaking of the uniqueness of solutions to \eqref{eq:main} is only meaningful if we identify the solutions $(n,1,a)$ and $(n,2,a)$. Alternatively one can demand to consider only those solutions $(n,m,a)$ with $n>m>1$, which we will do.

\begin{theorem}\label{th:main}
 Let $y>1$ be a fixed integer, then there exists at most one solution $(n,m,a)\in \Z^3$ to the Diophantine equation
 \begin{equation}\label{eq:main-gen}
   F_n+F_m=y^a, \qquad n>m>1,\;\; a>0,
 \end{equation}
 unless $y= 2,3,4,6,10$. In the case that $y=2,3,4,6$ or $10$ all solutions are listed below:
 \begin{description}
  \item[$y=2$] $(n,m,a)=(4,2,2),(5,4,3),(7,4,4)$;
  \item[$y=3$] $(n,m,a)=(3,2,1),(6, 2, 2)$;
  \item[$y=4$] $(n,m,a)=(4, 2, 1), (7, 4, 2)$;
  \item[$y=6$] $(n,m,a)=(5, 2, 1), (9, 3, 2)$;
  \item[$y=10$] $(n,m,a)=(6, 3, 1), (16, 7, 3)$.
 \end{description}

\end{theorem}

A direct consequence of this theorem is:

\begin{corollary}\label{cor:main}
 Let $p$ be a fixed prime, then the  Diophantine equation
 \begin{equation}\label{eq:main-p}
   F_n+F_m=p^a, \qquad n>m>1,\;\; a>0,
 \end{equation}
 has at most one solution $(n,m,a)\in \Z^3$, unless $p= 2$ or $p=3$.
\end{corollary}

We want to note that Equation \eqref{eq:main} has been recently studied by Luca and Patel~\cite{Luca:2018}, who found all solutions to \eqref{eq:main} with $a>1$, under the assumption that $n\equiv m \mod 2$. Furthermore, Kebli et.al. \cite{Kebli:2021} considered a similar but slightly different problem. They considered the Diophantine equation
\begin{equation}\label{eq:Kebli}
F_n\pm F_m=y^a \qquad n\geq m\geq 0,\;\; a\geq 2.
\end{equation}
Note that the essential difference is that they exclude solutions with $a=1$. In this case, they proved upper bounds for $n$ and $a$ depending on $y$ (which we will reproduce in Proposition \ref{prop:Bravo-Luca-bound}) and showed that there exists no solution if $2\leq y \leq 1000$ other than those solutions satisfying $0\leq m \leq n \leq 36$. Under the assumption that the $abc$-conjecture holds, they also showed that there exist at most finitely many solutions $(n,m,a,y)$ to \eqref{eq:Kebli}.

In view of the results due to Kebli et.al. \cite{Kebli:2021} we obtain the following result as a direct consequence from Theorem \ref{th:main}:

\begin{corollary}
 Assume that $y\neq 2,3,4,6,10$. The Diophantine equation
 $$F_n+F_m=y^a, \qquad n>m>0,\;\; a\geq 2$$
 has no solution $(n,m,a)$, if $y$ can be represented as the sum of two non-zero, distinct Fibonacci numbers.
\end{corollary}

In the next section we will gather several useful results. In particular, we state lower bounds for linear forms in logarithms due to Matveev \cite{Matveev:2000} and Laurent \cite{Laurent:2008}, the LLL-reduction method, Baker-Davenport reduction and also some results concerning continued fractions. In Section \ref{sec:Bravo} we will recap the strategy due to Bravo and Luca and find an upper bound for a solution $(n,m,a)$ to \eqref{eq:main-gen}, which depends on $y$. In order to apply lower bounds for linear forms in logarithms we have to ensure that these linear forms do not vanish. Therefore we have to discuss the multiplicative dependence of $\alpha=\frac{1+\sqrt{5}}2$, $\sqrt{5}$ and $\tau(t)=\frac{\alpha^t+1}{\sqrt 5}$ for positive integers $t$. This will be done in Section \ref{sec:mult-dep}. The following two sections (Sections \ref{sec:bound} and \ref{sec:reduction}) are the heart of the paper. Assuming that there exist two solutions $(n_1,m_1,a_1)$ and $(n_2,m_2,a_2)$, with $a_1<a_2$ we proof an absolute upper bound for $n_2$ in Section \ref{sec:bound}. Using continued fractions, and LLL-reduction we are left with $\sim 110000$ possible candidates $y$ such that Diophantine equation \eqref{eq:main-gen} has two solutions. In the final section we discuss an implementation of the method due to Bravo and Luca \cite{Bravo:2014} to find for each possible candidate all solutions to \eqref{eq:main-gen}. This will complete our proof of the main theorem, Theorem \ref{th:main}.

\section{Some useful tools}

First, let us state the well known Binet formula for the Fibonacci sequence
$$
F_n=\frac{\alpha^n-\beta^n}{\alpha-\beta}=\frac{\alpha^n-\beta^n}{\sqrt{5}},
$$
where $\alpha=\frac{1+\sqrt{5}}2$ and $\beta=\frac{1-\sqrt{5}}2$ are the characteristic roots
of the characteristic polynomial $X^2-X-1$ of the Fibonacci sequence. Let us note that $\beta=-\alpha^{-1}$.

The Binet formula immediately yields for $n>1$ the inequalities 
\begin{equation}\label{eq:Fib-ieq}
0.38 \alpha^n <\alpha^n \frac{1-\alpha^{-4}}{\sqrt 5}\leq F_n=\alpha^n\frac{1-(-1)^n \alpha^{-2n}}{\sqrt 5}\leq \alpha^n \frac{1+\alpha^{-6}}{\sqrt 5}< 0.48 \alpha^n.
\end{equation}

\begin{lemma}\label{lem:el-bounds}
 Assume that $n>m>1$, and that $F_n+F_m=y^a$ for some integer $y>1$. Then we have
 \begin{itemize}
  \item $0.38 \alpha^n<F_n+F_m<0.78 \alpha^n$ and
  \item $a<\frac{n \log \alpha}{\log y}<0.7 n$.
 \end{itemize}
\end{lemma}

\begin{proof}
 Due to \eqref{eq:Fib-ieq} we have
 $$ 0.38\alpha^n<F_n<F_n+F_m<0.48 \alpha^n+0.48\alpha^{n-1}<0.78 \alpha^n$$
 which proves the first statement. For the second statement we consider the inequality
 $$y^a=F_n+F_m<0.78 \alpha^n$$
 and take logarithms. This yields $a \log y< n\log \alpha$ and by noting that $n,y\geq 2$ we immediately get the second statement.
\end{proof}

At the end of the proof of Theorem \ref{th:main} we will use the concept of the so-called Zeckendorf expansion of a number. The famous theorem of Zeckendorf states that every integer $N$ can be written uniquely as a sum of non consecutive Fibonacci numbers, i.e.
$$N=\sum_{i=1}^k F_{d_i},$$
with $1<d_1<d_2<\dots<d_k$ and $d_{i}-d_{i-1}>1$ for all $i=2,\dots,k$. This expansion can be easily computed by a greedy digit algorithm. That is, we compute the Zeckendorf expansion inductively. Let $F_{d_k}\leq N$ be the largest Fibonacci number smaller or equal to $N$ and assume that we have already found the Zeckendorf expansion 
$$N-F_{d_k}= \sum_{i=1}^{k-1} F_{d_i},$$
of $N-F_{d_k}$, then
$$N=\sum_{i=1}^k F_{d_i}$$
is the Zeckendorf expansion of $N$. 

Furthermore let us note a simple fact form calculus. If $x\in \R$ satisfies $|x|<1/2$, then 
\begin{align*}
|\log(1+x)|&<|x-x^2/2+-\dots|\\
&<|x|+\frac{|x|^2+|x|^3+\dots}2\\
&<|x|\left(1+\frac{|x|}{2(1-|x|)}\right)<\frac 32 |x|.
\end{align*}
Similarly we obtain the lower bound $|\log(1+x)|>\frac 12 |x|$ provided that $|x|<1/2$. We will use these inequalities frequently throughout the paper.

In order to obtain absolute upper bounds we will apply results on lower bounds for linear forms in  logarithms. To state this results we need the notion of height. Therefore let $\alpha \neq 0$ be an algebraic number of degree $d$ and let
$$a_0(X-\alpha_1)\cdots (X-\alpha_d) \in \Z[X]$$
be the minimal polynomial of $\alpha$. Then the absolute logarithmic Weil height is defined by
$$h(\alpha)=\frac 1d \left(\log |a_0|+\sum_{i=1}^d \max\{0,\log|\alpha_i|\}\right).$$ 
With this basic notation we have the following result on lower bounds for linear forms in logarithms due to Matveev \cite{Matveev:2000}.

\begin{lemma}\label{lem:Matveev}
Denote by $\alpha_1,\dots ,\alpha_n$ algebraic numbers, $\neq 0,1$, by $\log \alpha_1,\dots,\log \alpha_n$ determinations of their logarithms, by $D$ the degree over $\Q$ of the number field $K=\Q(\alpha_1,\dots ,\alpha_n)$, and by $b_1,\dots, b_n$ rational integers. Furthermore let $\kappa= 1$ if $K$ is real and $\kappa= 2$ otherwise. For all integers $j$ with $1\leq j \leq n$ choose
$$A_j\geq h'(\alpha_j)=\max \{Dh(\alpha_j),|\log \alpha_j|,0.16\},$$ 
and set 
$$B= \max \left\{\{1\}\cup \{|b_j|A_j/A_n\: :\:  1\leq j \leq n\}\right\}.$$
Assume that
$$\Lambda=b_1\log \alpha_1 +\cdots+b_n\log\alpha_n\neq  0.$$
Then
$$\log|\Lambda|\geq - C(n,\kappa) \max\{1,n/6\}C_0 W_0 D^2\Omega$$
with
\begin{gather*}
 \Omega=A_1\cdots A_n\\
 C(n,\kappa)=\frac{16}{n! \kappa} e^n (2n+ 1 + 2\kappa)(n+ 2)(4(n+ 1))^{n+1}\left(\frac {en}2 \right)^\kappa,\\
 C_0= \log(e^{4.4n+7}n^{5.5}D^2\log(eD)),\quad  W_0= \log(1.5eBD\log(eD))
\end{gather*}
\end{lemma}

In all our applications we will be in the situation, where $K=\Q(\sqrt 5)$ and $n=3$ or $n=2$. In the case of three logarithms Matveev's lower bound is
$$\log |\Lambda|> - 7.26\cdot 10^{10}\Omega \log(13.81 B).$$
Also note that instead of $B$ one can use $B^*=\max_{1\leq j \leq n}\{b_j\}$ in Matveev's bound.

However, in the case that the number of logarithms is $n=2$ we have numerically rather good results due to Laurent \cite{Laurent:2008}. In particular, we will use the following result:

\begin{lemma}\label{lem:Laurent}
 Suppose that the numbers $\alpha_1$, $\alpha_2$, $\log \alpha_1$, $\log \alpha_2$ are  real and positive and that $\alpha_1$ and $\alpha_2$ are multiplicatively independent. Then for positive integers $b_1$ und $b_2$ we have
 $$\log |b_1\log \alpha_1 -b_2\log \alpha_2| > - 17.9 D^4 \left(\max\{\log b' + 0.38, 30/D, 1\}\right)^2 \log A_1 \log A_2,$$
 where $D=[\Q(\alpha_1,\alpha_2):\Q]$,
 $$\log A_i\geq \max\{h(\alpha_i),\log \alpha_i/D,1/D\}$$
 for $i=1,2$ and
 $$b'=\frac{b_1}{D\log A_2}+\frac{b_2}{D\log A_1}.$$
\end{lemma}

In all of our applications the algebraic numbers appearing in the linear form of logarithms $\Lambda$ will be $y$, $\alpha$, $\sqrt{5}$ and $\tau(t)=\frac{\alpha^t+1}{\sqrt 5}$ for some positive integer $t$. Thus we will compute the modified heights of these numbers:

\begin{lemma}\label{lem:heights}
 Let $t>0$ and $y>1$ be integers and let $K=\Q(\sqrt 5)$ be the base field. Then we have 
 \begin{itemize}
  \item $h'(y)=2\log y$,
  \item $h'(\sqrt 5)= \log 5$,
  \item $h'(\alpha)=\log \alpha$ and
  \item $h'\left(\tau(t)\right)\leq t \log \alpha +1.29<\max\{3,t\} $.
 \end{itemize}
\end{lemma}

\begin{proof}
 We start with computing the height of $y$, $\sqrt 5$ and $\alpha$. Since all three are algebraic integers we have $a_0=1$ and obtain $h(y)=\log y$, and $h(\sqrt 5)=\log \sqrt 5$ and $h(\alpha)=1/2 \log \alpha$. Considering their numeric values and the fact that $D=[K:\Q]=2$ we obtain the first three results.
 
 For the computation of the height of $\tau(t)=\frac{\alpha^t+1}{\sqrt 5}$ we note first, that the Galois conjugate of $\tau(t)$ is $\frac{\beta^t+1}{-\sqrt{5}}$ which has absolute value $<1$ for all $t>0$. Moreover the denominator of $\tau(t)$ is at most $5$. Thus we obtain
 $$h(\tau(t))\leq \frac 12 \left(\log 5+\log \frac{\alpha^t+1}{\sqrt{5}}\right).$$
 It is easy to check that twice the upper bound for $h(\tau(t))$ is larger than $0.16$. Therefore
 we get
 $$h'(\tau(t))\leq \log \sqrt{5}+\log(\alpha^t+1)= t\log \alpha+\log(\sqrt 5 (1+\alpha^{-t}))<
 t\log\alpha + 1.29.$$
 From this bound it is easy to deduce that $h'(\tau(t))\leq \max\{3,t\}$.
\end{proof}

Let $\LL\subseteq \R^k$ be a $k$-dimensional lattice with LLL-reduced basis
$b_1,\dots,b_k$ and denote by $B$ be the matrix with columns $b_1,\dots, b_k$. Moreover, we denote by $b^*_1,\dots,b^*_k$ the orthogonal basis of $\R^k$ which we obtain
by applying the Gram-Schmidt process to the basis $b_1,\dots,b_k$. In particular, we have that
$$b^*_i=b_i-\sum_{j=1}^{i-1}\mu_{i,j}b^*_j, \qquad \mu_{i,j}=\frac{\langle b_i,b_j\rangle}{\langle b_j^*,b_j^*\rangle}.$$
Further, let us define
\begin{equation*}
l(\LL,y)=\begin{cases}
\min_{x \in \LL} \|x-y\|, & y \not\in \LL, \\
\min_{0 \neq x \in \LL} \|y\|, & y \in \LL,
\end{cases}
\end{equation*}
where $\|\cdot\|$ denotes the euclidean norm on $\R^k$. It is well known,
that by applying the LLL-algorithm it is possible to give in polynomial time
a lower bound for $l(\LL,y) \geq  c_1$ (see e.g. \cite[Section V.4]{Smart:DiGL}):

\begin{lemma}\label{lem:lattice}
 Let $y\in \R^k$ and $z=B^{-1}y$ with $z=(z_1,\dots,z_k)^T$. Furthermore we define:
 \begin{itemize}
 \item If $y\not\in\LL$ let $i_0$ be the largest index such that $z_{i_0}\neq 0$ and put $\sigma=\{z_{i_0}\}$, where $\{\cdot\}$ denotes the distance to the nearest integer.
 \item If $y\in \LL$ we put $\sigma=1$.
 \end{itemize}
 Finally let
 $$c_2=\max_{1\leq j\leq k}\left\{\frac{\|b_1\|^2}{\|b_j^*\|^2} \right\}.$$
 Then we have
 $$l(\LL,y)^2\geq c_2^{-1}\sigma^2 \|b_1\|^2=c_1^2.$$
\end{lemma}

In our applications we suppose that we are given real numbers $\eta_0,\eta_1,\dots,\eta_k$  linearly independent over $\Q$ and two positive constants
$ c_3, c_4$ such that
\begin{equation} \label{eq:redform1}
|\eta_0+x_1\eta_1+\dots+x_k\eta_k| \le  c_3\exp(- c_4H),
\end{equation}
where the integers $x_i$ are bounded by $|x_i| \leq X_i$ with $X_i$ given upper bounds for $1 \leq i \leq k$.
We write $X_0=\max_{1 \leq i \leq k}\{X_i\}$. The basic idea in such a situation, due to de Weger \cite{deWeger:1987}, is to approximate the linear form
\eqref{eq:redform1} by an approximation lattice. Namely, we consider the lattice $\LL$ generated by the columns of the matrix
\begin{equation*}
\mathcal{A}=\begin{pmatrix}
    1 & 0 & \dots & 0  & 0 \\
    0 & 1 & \dots & 0  & 0 \\
    \vdots & \vdots & \vdots & \vdots & \vdots \\
    0 & 0 & \dots & 1  & 0 \\
    \lfloor{ C\eta_1}\rfloor & \lfloor{ C\eta_2}\rfloor & \dots & \lfloor{ C\eta_{k-1}}\rfloor & \lfloor{ C\eta_k}\rfloor
\end{pmatrix}
\end{equation*}
where $C$ is a large constant usually of the size of about $X_0^k$. Let us assume that we have an LLL-reduced basis $b_1,\dots,b_k$ of $\LL$ and that we have a lower bound
$l(\LL,y)\geq c_1$ with $y=(0,0,\dots ,-\lfloor{C\eta_0}\rfloor)$. Note that $c_1$ can be computed by using the results of Lemma \ref{lem:lattice}. Then we have with these notations the following lemma (e.g. see \cite[Lemma VI.1]{Smart:DiGL}):

\begin{lemma}\label{lem:real-reduce}
Assume that $S=\sum_{i=1}^{k-1}X_i^2$ and $T=\frac{1+\sum_{i=1}^k{X_i}}{2}$. If $c_1^2 \geq T^2+S$, then inequality \eqref{eq:redform1} implies that
we have either $x_1=x_2=\dots=x_{k-1}=0$ and $x_k=-\frac{\lfloor{ C\eta_0}\rfloor)}{\lfloor{ C\eta_k}\rfloor)}$ or
\begin{equation} \label{eq:reduction-real}
H \leq \frac{1}{ c_4}\left(\log(C c_3)-\log\left(\sqrt{ c_1^2-S}-T\right) \right).
\end{equation}
\end{lemma}

In order to reduce the huge bounds coming from Matveev's or Laurent's result we also use continued fractions. In particular, we use the following two results: 

\begin{lemma}\label{lem:cont-frac}
 Assume that $\mu$ is real and irrational and has the continued fraction expansion $\mu=[a_0;a_1,a_2,\dots]$. Let $\ell$ be an integer and set $A=\max_{1\leq j\leq \ell}\{a_j\}$ and let $p_\ell/q_\ell$ be the $\ell$-th convergent to $\mu$, then
 $$\frac{1}{(2 + A)q_\ell} < \left| q \mu - p \right|$$
 for any integer $q$ with $|q|\leq q_\ell$.
\end{lemma}

\begin{proof}
 This follows from the inequality given in \cite[page 47]{Baker:NT} combined with the best approximation property of continued fractions.
\end{proof}

The second method is due to Baker and Davenport \cite{Baker:1969} for which we state a variant of this reduction method:

\begin{lemma}\label{lem:BakDav1}
Given a Diophantine inequality of the form
\begin{equation}\label{IEq:BakerDav}
|n\mu+\tau-x|<c_1\exp(-c_2 H),
\end{equation}
with positive constants $c_1,c_2$ and real numbers $\mu$ and $\tau$ and integers $n$ and $x$.
Assume $n<N$ and that there is a real number $\kappa>1$ such that there exists a convergent $p/q$ to $\mu$ with
$$\{q \mu\}<\frac 1{2 \kappa N} \quad \text{and} \quad \{q \tau\}>\frac 1{\kappa},$$
where $\{\cdot\}$ denotes the distance to the nearest integer. Then we have
$$H\leq\frac{\log (2 \kappa q c_1)}{c_2}.$$
\end{lemma}

\begin{proof}
We consider inequality \eqref{IEq:BakerDav} and multiply it by $q$. Then under our assumptions we obtain
\begin{equation}\label{eq:BD-proof}
\begin{split}
c_1q\exp(-c_2 H)&>|qx+nq\mu+q\tau|
\geq \left|\{q\tau\}-\{Nq\mu\}\right|\\
&=\left|\{q\tau\}-N\{q\mu\}\right|> \frac1{2\kappa}.
\end{split}
\end{equation}
Note that the equality in \eqref{eq:BD-proof} holds since by assumption $N\{q\mu\}<\frac{1}{2\kappa}<\frac 12$ and therefore $N\{q\mu\}=\{Nq\mu\}$. Solving Inequality \eqref{eq:BD-proof} for $H$ yields the lemma.
\end{proof}

\section{The method of Bravo and Luca}\label{sec:Bravo}

In this section we give a recap of the method of Bravo and Luca \cite{Bravo:2016}  and bring their result in a form which is suitable for us. In particular, the following proposition is essentially a slightly improved version of Theorem 1 in \cite{Kebli:2021}.

\begin{proposition}\label{prop:Bravo-Luca-bound}
 Let $(n,m,a)$ be a solution to \eqref{eq:main-gen} for a fixed integer $y>1$, then 
 $$n< 3.4 \cdot 10^{22} (\log y)^2 (\log (13.81 n))^2$$
\end{proposition}

We start by rewriting \eqref{eq:main-gen} and obtain
$$ \frac{\alpha^n}{\sqrt 5}-y^a=\frac{-\alpha^m+\beta^n+\beta^m}{\sqrt 5}.$$
Dividing through $y^a=F_n+F_m$ and estimating the right hand side yields
$$
 \left|\frac{\alpha^n}{y^a\sqrt{5}}-1\right|<\frac{\alpha^m+\alpha^{-n}+\alpha^{-m}}{0.38 \alpha^n \sqrt{5}}<\alpha^{m-n} \frac{1+\alpha^{-6}+\alpha^{-5}}{0.38 \sqrt{5}}< 1.35 \alpha^{n-m}
$$
Let us assume for the moment that $n-m\geq 3$. Then taking logarithms yields
\begin{equation}\label{eq:Lambda1}
|\Lambda_1|= \left|n\log \alpha -a \log y +\log \sqrt{5}\right|<2.03 \alpha^{m-n}
\end{equation}
Note that $\Lambda_1=0$ would imply that $\alpha^m-\beta^n-\beta^m=0$. But then we have 
$$0.38<\alpha -2\alpha^{-1}<\alpha^m-\beta^n-\beta^m=0$$
which is an absurdity. Moreover we note that $B^*=n$ in this case. Thus we obtain with our choice of $A_1=\log \alpha $, $A_2=2\log y $ and $A_3=\log 5$ from Lemma \ref{lem:heights} the inequality
$$ (n-m)\log \alpha - 0.71<-\log |\Lambda_1|<7.26\cdot 10^{10} (2 \log y)( \log 5)( \log \alpha) \log(13.81 n),$$
which yields
\begin{equation}\label{eq:n-m_ieq}
n-m< 2.34 \cdot 10^{11} \log y \log(13.81 n). 
\end{equation}
We want to note that this upper bound for $n-m$ exceeds by far our assumption that $n-m\geq 3$. Thus \eqref{eq:n-m_ieq} holds in any case.

Again we rewrite Equation \eqref{eq:main-gen} and obtain
$$ \frac{\alpha^n+\alpha^m}{\sqrt{5}}-y^a=\alpha^m\tau(n-m)-y^a=\frac{\beta^n+\beta^m}{\sqrt 5}.$$
We divide through $y^a=F_n+F_m$ and obtain
$$\left|\frac{\alpha^m \tau(n-m)}{y^a}-1\right|<\frac{\alpha^{-3}+\alpha^{-2}}{0.38 \alpha^n \sqrt{5}}< 0.73 \alpha^{-n}$$
Since $n>m>1$ the right hand side of the inequality is always $<0.5$ and we obtain
\begin{equation}\label{eq:Lambda2}
|\Lambda_2|= \left|n\log \alpha -a \log y +\log \tau(n-m) \right|<1.1 \alpha^{-n}.
\end{equation}
First, we note that $\Lambda_2=0$ would imply that $\beta^n+\beta^m=0$ which is impossible. According to Lemma \ref{lem:heights} we choose $A_1=\log \alpha$, $A_2=2\log y$ and $A_3=\max\{3,n-m\}$. Thus an application of Matveev's lower bound  yields
\begin{equation}\label{eq:Lambda2-ieq}
n \log \alpha - 0.1 <-\log |\Lambda_2|<7.26\cdot 10^{10} (2 \log y) \max\{3,n-m\} \log \alpha \log(13.81 n).
\end{equation}
In view of the upper bound \eqref{eq:n-m_ieq} for $n-m$ we get
$$ n< 3.4 \cdot 10^{22} (\log y)^2 (\log (13.81 n))^2,$$
which proves Proposition \ref{prop:Bravo-Luca-bound}.

If $y$ is a fixed integer it is easy to find now an absolute bound for $n$. However, for a proof of Theorem \ref{th:main} we have to find a bound independent of $y$.

\section{Multiplicative independence results}\label{sec:mult-dep}

Before we start with the main part of the proof of our results we discuss the multiplicative independence of $\alpha$, $\sqrt{5}$ and $\tau(t)$. First, note that the cases $t=1,2$ and $10$ yield multiplicative dependence relations since 
\begin{align*}
\tau(1)=&\frac{\alpha+1}{\sqrt{5}}=\frac{\alpha^2}{\sqrt 5},\\
\tau(2)=&\frac{\alpha^2+1}{\sqrt{5}}=\alpha,\\
\tau(10)=&\frac{\alpha^{10}+1}{\sqrt{5}}=5 \alpha^5.
\end{align*}
However, these are essentially the only obstructions we have to the multiplicative independence of $\alpha$, $\sqrt{5}$, $\tau(t_1)$ and $\tau(t_2)$. In particular, the main result of this section is the following theorem that might be also of independent interest:

\begin{theorem}\label{th:independence}
 Let $t>0$ and $t_2>t_1>0$ be integers. Then the following holds:
 \begin{itemize}
  \item The algebraic numbers $\alpha$ and $\tau(t)$ are multiplicatively independent unless $t=2$.
  \item The algebraic numbers $\tau(t_1)$ and $\tau(t_2)$ are multiplicatively independent.
  \item The algebraic numbers $\alpha,\tau(t_1)$ and $\tau(t_2)$ are multiplicatively independent unless $t_1=2$, or $t_2=2$, or $(t_1,t_2)=(1,10)$.
  \item The algebraic numbers $\alpha,\sqrt{5}$ and $\tau(t)$ are multiplicatively independent unless $t=1,2$ or $10$.
 \end{itemize}
\end{theorem}

Our application of Theorem \ref{th:independence} is the following:

\begin{corollary}\label{cor:non-vanishing}
 Let
 \begin{align*}
 \Lambda_{4a}&= \Delta \log \alpha-a_1\log \sqrt{5}-a_2\log \tau(t)\\
 \Lambda_{4b}&= \Delta \log \alpha-a_2\log \sqrt{5}-a_1\log \tau(t)\\
  \Lambda_5 &= \Delta \log \alpha-a_1\log \tau(t_1)+a_2\log \tau(t_2)
 \end{align*}
 with integers $t_1,t_2>0$ and integers $a_1,a_2,\Delta$, with $0<a_1<a_2$. Then we have
 \begin{itemize}
 \item $\Lambda_{4a}\neq 0$;
 \item $\Lambda_{4b}\neq 0$; 
  \item $\Lambda_5\neq 0$, unless $t_1=t_2=2$ and $\Delta+a_2-a_1=0$.
\end{itemize}
\end{corollary}

\begin{proof}[Proof of Corollary \ref{cor:non-vanishing}]
If $t\neq 1,2,10$, then $\log\alpha,\log \sqrt{5},\log \tau(t)$ are linearly independent over $\Q$, thus $\Lambda_{4a},\Lambda_{4b} \neq 0$ unless $t=1,2,10$.

Let us consider the linear form $\Lambda_{4a}$. In the case that $t=1$ we obtain
$$\Lambda_{4a}=(\Delta-2a_2)\log \alpha-(a_1-a_2)\log \sqrt{5}.$$
Since $\log \alpha$ and $\log \sqrt 5$ are linearly independent over $\Q$, we deduce that $\Lambda_{4a}=0$ if and only if $\Delta-2a_2=a_1-a_2=0$. But, we assume that $a_2>a_1$, thus $a_1-a_2\neq 0$ and therefore $\Lambda_{4a}\neq 0$.

In the case that $t=2$ we get
$$\Lambda_{4a}=(\Delta-a_2)\log \alpha-a_1\log \sqrt{5}.$$
By the linear independence of $\log \alpha$ and $\log \sqrt 5$, and by the assumption that $a_1\neq 0$ we deduce that $\Lambda_{4a}\neq 0$ holds also in this case.

Finally, we consider the case that $t=10$. We get
$$\Lambda_{4a}= (\Delta-5a_1) \log \alpha+(-a_1-2a_2)\log \sqrt{5}$$
and since $a_1+2a_2\neq 0$ we deduce similarly as above that also in this case $\Lambda_{4a}\neq 0$ holds.

The proof that $\Lambda_{4b}\neq 0$ in the special cases that $t=1,2$ or $10$ is almost identical and will be omitted.

 If $t_1\neq t_2$ or $t_1,t_2\neq 2$ or $(t_1,t_2)=(1,10)$ or $(t_1,t_2)=(10,1)$, then $\log \alpha,\log \tau(t_1)$ and $\log \tau(t_2)$ are linearly independent over $\Q$ and thus $\Lambda_5\neq 0$, since we assume that $a_1,a_2\neq 0$.
 
 In the case that $t_1=t_2=t$ we get
 $$\Lambda_5=\Delta \log \alpha -(a_1-a_2)\log \tau(t).$$
 Assume for the moment that $t\neq 2$. Since in this case $\log\alpha$ and $\log\tau(t)$ are linearly independent over $\Q$ and since $a_1\neq a_2$ we obtain that $\Lambda_5\neq 0$. Therefore we consider the case that $t_1=t_2=2$. In this case we have
$$\Lambda_5=(\Delta+a_2-a_1)\log \alpha$$
and $\Lambda_5=0$ if and only if $\Delta+a_2-a_1=0$.  
 
 In the case that $t_1=2$ but $t_2\neq 2$ we have
 $$\Lambda=(\Delta+a_2)\log \alpha- a_1\log \tau(t_2).$$
Since $\log \alpha$ and $\log \tau(t_2)$ are linearly independent provided that $t_2\neq 2$ and since $a_1\neq 0$ by assumption we deduce that $\Lambda_5\neq 0$. A similar argument applies to the case that $t_2=2$ and $t_1\neq 2$. 

In the case that $t_1=1$ and $t_2=10$ we obtain
$$\Lambda_5=(\Delta -5a_1+2a_2)\log \alpha - (2a_1+a_2)\log \sqrt 5.$$
Since we assume that $a_2,a_1>0$ and since $\log \alpha$ and $\log \sqrt 5$ are linearly independent over $\Q$ we deduce that $\Lambda_5\neq 0$ in this case.

In the case that $t_1=10$ and $t_2=1$ we obtain
$$\Lambda_5=(\Delta +5a_2-2a_1)\log \alpha + (2a_2+a_1)\log \sqrt 5$$
and by the linear Independence of $\log \alpha$ and $\log \sqrt 5$ over $\Q$ we deduce that $\Lambda_5\neq 0$ holds as well.
\end{proof}

Now, let us turn to the proof of Theorem \ref{th:independence}. We start by proving that $\alpha$ and $\tau(t)$ are multiplicatively independent unless $t=2$. This can be seen as follows. Since $\alpha$ is a fundamental unit in $K=\Q(\sqrt 5)$ we have to show that $\tau(t)$ is not a unit in $K$. Or in other words we have to show that
\begin{equation}\label{eq:indep-eq-1}
\frac{\alpha^t+1}{\sqrt 5}=\tau(t)=\pm \alpha^x, \qquad t>0,\; x\in \Z
\end{equation}
has no solution other than $t=2$ and $x=1$.
Taking the norm on both sides of \eqref{eq:indep-eq-1} we obtain
$$
\frac{\alpha^t+1}{\sqrt 5} \cdot \frac{\beta^t+1}{\sqrt 5}=\frac{(\alpha \beta)^t+\alpha^t+\beta^t+1}5=\frac{L_t+(-1)^t+1}{5}=\pm 1,
$$
where $L_n$ denotes the $n$-th Lucas number which is defined by $L_n=\alpha^n+\beta^n$.
Note that
$$L_n\geq \alpha^n-1,$$
hence $L_n>7$ for all $n>4$. Therefore equation \eqref{eq:indep-eq-1} has a solution only if $t\leq 4$. Since there is no solution for $t=1,3$ or $t=4$, we have shown that indeed $(t,x)=(2,1)$ is the only solution to \eqref{eq:indep-eq-1}. 

A key argument in the proof of Theorem \ref{th:independence} is the following result on primitive divisors due to Schinzel \cite{Schinzel:1993}:

\begin{theorem}[Schinzel \cite{Schinzel:1993}]
 Let  $K$  be an  algebraic number field, $A,  B$  integers of  $K$,  $\gcd(A,  B)  =1$, $AB\neq  0$,   $A/B$ of degree $d$   not  a  root  of  unity, and $\zeta_k$  a  primitive $k$-th root of unity in  $K$.  For every $\epsilon > 0$  there exists a constant $c(d,  \epsilon)$  such that if $n >  c(d, \epsilon)(1+\log k)^{1+\epsilon}$, there exists a prime ideal of  $K$  that divides $A^n-\zeta_k B^n$, but does not divide $A^m - \zeta_k^j B^m$ for  $m <  n$  and  arbitrary $j$.
\end{theorem}

If we choose $A=\alpha$, $B=1$, $k=2$ and $j=1$, then Schinzel's result implies that there exists an absolute constant $c$ with the following property. If $n>c$ is an integer, then there exists a  prime ideal $\mathfrak P$ of $K$ that divides $\alpha^n+1$ but does not divide $\alpha^m+1$ for any $m<n$. 

Thus a first step to prove Theorem \ref{th:independence} will be to compute the constant $c$ explicitly. Unfortunately the paper of Schinzel does not provide any explicit formula for $c(\epsilon,k)$ and therefore also not for $c$. Thus we have to go through Schinzel's proof to find a concrete value for $c$ in our case.

First, we note that a key argument for Schinzel's result is the following lemma (\cite[Lemma 4]{Schinzel:1974} see also \cite[Lemma 3]{Schinzel:1993}) on primitive divisors. A prime ideal $\mathfrak P$ is called a primitive divisor of $A^n-B^n$, if $\mathfrak P$ divides $A^n-B^n$ but does not divide any $A^m-B^m$ with $m<n$.

\begin{lemma}\label{lem:Schinzel}
 Let  $\phi_n(x,y)$ be the  $n$-th cyclotomic polynomial in homogeneous form. If $\mathfrak P$ is a prime ideal of $K$,  $n>  2(2^d -1)$,  $\mathfrak P | \phi_n (A,  B)$,  and  $\mathfrak P$ is not a primitive divisor of $A^n -  B^n$, then
 \begin{equation}\label{eq:Schinzel-crit}
 \ord_{\mathfrak P} (\phi_n(A, B)) < \ord_{\mathfrak P} (n).
 \end{equation}
\end{lemma}

Assume for the moment that every prime ideal $\mathfrak P$ which lies over a rational prime $p$ satisfies 
$$\frac{\ord_p \left(\Norm_{K/\Q} (\alpha^n+1)\right)}2=\ord_{\mathfrak P}(\alpha^n+1)<\ord_{\mathfrak P}(2n)\leq \ord_p (4n^2).$$
Then we obtain the inequality
$$\alpha^n-1<\Norm_{K/\Q}(\alpha^n+1)=\alpha^n+\beta^n+1+(-1)^n <8 n^2,$$
which implies $n\leq 15$. Therefore let us assume that $n>15$. Our previous computation implies that in this case there exists a prime ideal $\mathfrak P$ that divides $\alpha^n+1$ and
for which $\ord_{\mathfrak P}(\alpha^n+1)<\ord_{\mathfrak P}(2n)$ does not hold. Therefore $\mathfrak P$ divides
$$\alpha^{2n}-1=(\alpha^n+1)(\alpha^n-1)$$
and Lemma \ref{lem:Schinzel} implies that $\mathfrak P$ is a primitive divisor of $\alpha^{2n}-1$.
Thus $\mathfrak P$ is not a divisor of any algebraic number of the form$\alpha^{2m}-1=(\alpha^m+1)(\alpha^m-1)$ and we have proved the following lemma.

\begin{lemma}
 The number $\alpha^n+1$ has a primitive divisor if $n>15$.
\end{lemma}

Now we will prove Theorem \ref{th:independence}. First, we note that since the prime ideal $(\sqrt 5)$ divides $\tau(10)$, the numbers $\tau(t)$ have primitive divisors, provided $t>15$. Moreover these primitive divisors are not $(\sqrt 5)$. Since $\tau(t)$ is not a unit unless $t=2$, Theorem~\ref{th:independence} immediately follows in the case that $t_2>15$.

In the case that $t_2\leq 15$ a brute force search reveals all multiplicative dependencies listed in Theorem \ref{th:independence}, which proves the theorem.

\section{An absolute bound for $n_1$}\label{sec:bound}

Let us assume that there exist two solutions $(n_1,m_1,a_1)$ and $(n_2,m_2,a_2)$ with $a_1<a_2$ to \eqref{eq:main-gen}. Let us note that by our assumption that $1<m<n$ and the uniqueness of the Zeckendorf expansion we deduce that two solutions $(n_1,m_1,a_1)$ and $(n_2,m_2,a_2)$ always yield $a_1\neq a_2$. We start with the following lemma:

\begin{lemma}
 With the assumptions from above we have $n_1<n_2$.
\end{lemma}

\begin{proof}
Assume to the contrary that $n_1\geq n_2$  and let $y^{a_1}=F_{n_1}+F_{m_1}$ and $y^{a_2}=F_{n_2}+F_{m_2}$. We obtain
$$y(F_{n_2}+F_{m_1})\leq y(F_{n_1}+F_{m_1})=y^{a_1+1}\leq y^{a_2}=F_{n_2}+F_{m_2}$$
which explains the second inequality in
$$F_{n_2}\leq (y-1)F_{n_2}+yF_{m_1}\leq F_{m_2}.$$
However the inequality above yields an obvious contradiction and we conclude that indeed $n_1<n_2$ holds.
\end{proof}

Let us note that for each solution $(n_i,m_i,a_i)$ with $i=1,2$ to \eqref{eq:main-gen} we obtain two linear forms in logarithms of the form \eqref{eq:Lambda1} and \eqref{eq:Lambda2}. That is we have four linear forms in logarithms
\begin{align}
 \label{eq:Lambda11} \Lambda_{11} &=n_1\log \alpha -a_1 \log y -\log \sqrt{5}\\
 \label{eq:Lambda12} \Lambda_{12} &=n_2\log \alpha -a_2 \log y -\log \sqrt{5}\\
 \label{eq:Lambda21} \Lambda_{21} &=m_1 \log \alpha -a_1 \log y +\log \tau(n_1-m_1)\\
 \label{eq:Lambda22} \Lambda_{22} &=m_2 \log \alpha -a_2 \log y +\log \tau(n_2-m_2)
\end{align}
which have upper bounds
\begin{align*}
 |\Lambda_{11}|&< 2.03 \alpha^{m_1-n_1},& |\Lambda_{12}|&<2.03 \alpha^{m_2-n_2} ,\\
 |\Lambda_{21}|&< 1.1 \alpha^{-n_1},& |\Lambda_{22}|&<1.1 \alpha^{-n_2} ,
\end{align*}
respectively.

Our strategy to obtain an absolute upper bound for $n_1$, i.e. to obtain an upper bound not depending on $y$ is to combine these four linear forms in logarithms such that $\log y$ is eliminated.

We start with multiplying \eqref{eq:Lambda11} by $a_2$ and \eqref{eq:Lambda12} by $a_1$, subtract the resulting linear forms and obtain
\begin{equation}\label{eq:Lambda3}
 |\Lambda_3|=|\Delta\log \alpha - (a_2-a_1)\log \sqrt{5}|<2.03\left(\frac{a_2}{\alpha^{n_1-m_1}}+\frac{a_1}{\alpha^{n_2-m_2}}\right),
\end{equation}
with $\Delta=n_1a_2-n_2a_1$. Since $\alpha$ and $\sqrt{5}$ are multiplicatively independent and since $a_1\neq a_2$, we deduce that $\Lambda_3\neq 0$. Due to Lemma \ref{lem:el-bounds} we know that
$a_2<\frac{n_2 \log \alpha}{\log y}< 0.7 n_2$. Therefore we have
$$|\Lambda_3|<\frac{2.85 n_2}{\min\{\alpha^{n_1-m_1},\alpha^{n_2-m_2}\}}.$$

We want to apply Laurent's result, Theorem \ref{lem:Laurent}, to \eqref{eq:Lambda3} and note that $a_2-a_1,\Delta<n_2a_2<0.7 n_2^2$. Moreover, we choose
$$\log A_1=0.55> \max\left\{h(\alpha),\frac{\log \alpha}2,1/2\right\}$$
and
$$\log A_2=0.81>\max\left\{h(\sqrt{5}),\frac{\log \sqrt{5}}2,1/2\right\}$$
which results into
$$b'=\frac{\Delta}{2A_2}+\frac{a_2-a_1}{2A_1}<n_2^2 \exp(-0.38)$$
and therefore we have $\log b'+0.38<2\log n_2$.
Let us assume for the moment that $n_2\geq 1809>\exp(7.5)$ then we obtain
$$\max\{\log b'+0.38,15,1\}<2\log n_2$$
and therefore
$$ -510.37 (\log n_2)^2<\log|\Lambda_3|<-\log \alpha \min\{n_1-m_1,n_2-m_2\} +\log n_2+\log 2.85.$$
This yields
$$\min\{n_1-m_1,n_2-m_2\}<1061 (\log n_2)^2.$$

We distinguish now between two cases.

\subsection{The case that $n_1-m_1=\min\{n_1-m_1,n_2-m_2\}$}

In this case we consider the linear forms \eqref{eq:Lambda12} and \eqref{eq:Lambda21}. By eliminating $\log y$ we obtain
$$
\Lambda_{4a}=\Delta \log \alpha+a_1\log \sqrt{5}-a_2\log \tau(n_1-m_1),
$$
with $\Delta=n_2a_1-m_1a_2$. Moreover, we obtain
\begin{equation}\label{eq:Lambda4a-ieq}
|\Lambda_{4a}|<\frac{1.1 a_2}{\alpha^{n_1}}+\frac{2.03 a_1}{\alpha^{n_2-m_2}}<\max\left\{\frac{2.2a_2}{\alpha^{n_1}},\frac{4.06a_1}{\alpha^{n_2-m_2}}\right\}.
\end{equation}
Since Corollary \ref{cor:non-vanishing} we have $\Lambda_{4a}\neq 0$ and therefore we get by applying Lemmas \ref{lem:Matveev} and \ref{lem:heights} the inequality
\begin{align*}
|\Lambda_{4a}|&>- 7.26\cdot 10^{10} h'(\alpha)h'(\sqrt 5)h'(\tau(n_1-m_1)) \log(9.67 n_2^2)\\
&>-5.97\cdot 10^{13}\log(3.11 n_2)^3.
\end{align*}
Note that $B^*=\max\{n_2a_1-m_1a_2,a_2,a_1\}<0.7 n_2^2$. This lower bound combined with the upper bound found for $\Lambda_{4a}$ yields
$$\min\{n_1,n_2-m_2\}< 1.25\cdot 10^{14}\log(3.11 n_2)^3. $$

\subsection{The case that $n_2-m_2=\min\{n_1-m_1,n_2-m_2\}$}

Similarly as in the case above we consider the combination of two linear forms. In this case we pick the linear forms \eqref{eq:Lambda11} and \eqref{eq:Lambda22}. Eliminating $\log y$ yields
$$
\Lambda_{4b}=\Delta \log \alpha+a_2\log \sqrt{5}-a_1\log \tau(n_2-m_2),
$$
with $\Delta=n_1a_2-m_2a_1$. Moreover, we obtain
$$
|\Lambda_{4b}|<\frac{1.1 a_1}{\alpha^{n_2}}+\frac{2.03 a_2}{\alpha^{n_1-m_1}}<\frac{4.06 a_2}{\alpha^{n_1-m_1}}.
$$
Since Corollary \ref{cor:non-vanishing} we have $\Lambda_{4b}\neq 0$. Therefore we can apply Lemmas \ref{lem:Matveev} and \ref{lem:heights} and get the inequality
\begin{align*}
|\Lambda_{4b}|&>- 7.26\cdot 10^{10} h'(\alpha)h'(\sqrt 5)h'(\tau(n_2-m_2)) \log(9.67 n_2^2)\\
&>-5.97\cdot 10^{13}\log(3.11 n_2)^3.
\end{align*}
Note that $B^*=\max\{n_1a_2-m_2a_1,a_2,a_1\}<0.7 n_2^2$. This lower bound combined with the upper bound found for $\Lambda_{4b}$ yields
$$n_1-m_1< 1.25\cdot 10^{14}\log(3.11 n_2)^3. $$

Thus we have in any case the following lemma.

\begin{lemma}\label{lem:ni-mi-bound}
 We have
 $$\min\{n_1-m_1,n_2-m_2\}<1061 (\log n_2)^2$$
 and one of the two following inequalities holds:
 \begin{itemize}
\item $ \max\{n_1-m_1,n_2-m_2\}< 1.25\cdot 10^{14}(\log(3.11 n_2))^3$, or
\item $n_1<1.25\cdot 10^{14}(\log(3.11 n_2))^3.$
\end{itemize}
\end{lemma}

\subsection{A bound for $n_1$}

Let us assume that we have found a bound for $\max\{n_1-m_1,n_2-m_2\}$ such as in the lemma above. We want to deduce an upper bound for $n_1$. To obtain such a bound we multiply \eqref{eq:Lambda21} by $a_2$ and \eqref{eq:Lambda22} by $a_1$ and subtract the resulting linear forms, which results in
\begin{equation*}
\begin{split}
 |\Lambda_5|&=|\Delta\log \alpha - a_1\log \tau(n_2-m_2)+a_2 \log \tau(n_1-m_1)|
 <1.1\left(\frac{a_2}{\alpha^{n_1}}+\frac{a_1}{\alpha^{n_2}}\right)\\
 &<2.2 \frac{a_2}{\alpha^{n_1}},
 \end{split}
\end{equation*}
with $\Delta=m_1a_2-m_2a_1$. Note that unless $n_1-m_1=n_2-m_2=2$ and $\Delta-(a_1-a_2)=0$ we have $\Lambda_5\neq 0$ and we can apply Matveev's lower bound (Lemma \ref{lem:Matveev}). Then we obtain
\begin{align}\label{eq:Lambda5}
-\log|\Lambda_5|&>-7.26\cdot 10^{10} h'(\alpha)h'(\tau(n_1-m_1)h'(\tau(n_2-m_2))\log(9.67 n_2^2)\\ \nonumber
&>-9.27\cdot 10^{27}(\log(3.11 n_2))^6
\end{align}
If we compare this with the upper bound for $|\Lambda_5|$ we obtain
\begin{equation}\label{eq:bound-n_1}
n_1< 1.93\cdot 10^{28}(\log(3.11 n_2))^6
\end{equation}
which is a larger upper bound for $n_1$, than the bound given in the second alternative of Lemma \ref{lem:ni-mi-bound}.

Now let us assume for the moment that $n_1-m_1=n_2-m_2=2$ and $\Delta-(a_1-a_2)=0$. We want to give an elementary lower bound for $\Lambda_{21}$ and $\Lambda_{22}$ in this case. Therefore we consider equation \eqref{eq:main-gen} with $n-m=2$ and obtain
$$\frac{\alpha^{n-1}}{y^a}-1=\frac{\alpha^n+\alpha^{n-2}}{y^a\sqrt 5}-1=\frac{\beta^{n}+\beta^{n-2}}{y^a\sqrt{5}}=\frac{\beta^{n-1}}{y^a}.$$
Furthermore, we know that $y^a=F_n+F_{n-2}$ and therefore we obtain the inequality
$$ 0.52\alpha^n <0.38 (\alpha^n+\alpha^{n-2}) < y^a<0.48 (\alpha^n+\alpha^{n-2})<0.67 \alpha^n.$$
Taking logarithms and applying upper and lower bounds for $y^a$ we obtain
$$ 0.26 \alpha^{-2n+1}<|(n-1)\log \alpha-a\log y |<1.01 \alpha^{-2n+1}.$$
Let us note that
$$\Lambda_2=(n-1)\log \alpha-a\log y$$
in the case that $n-m=2$ and therefore we have found lower and upper bounds for $|\Lambda_{21}|$ and $|\Lambda_{22}|$. 

If we multiply $\Lambda_{21}$ by $a_2$ and multiply $\Lambda_{22}$ by $a_1$ and subtract we obtain a lower bound for $\Lambda_5$ namely
$$ a_2\left(\frac{0.26}{\alpha^{2n_1-1}}-\frac{1.01}{\alpha^{2n_2-1}}\right)<\frac{0.26 a_2}{\alpha^{2n_1-1}}-\frac{1.01 a_1}{\alpha^{2n_2-1}}<|\Lambda_5|.$$
Therefore we conclude that $\Lambda_5\neq 0$ provided that $\alpha^{2n_2-2n_1}<1.01/0.26<3.89$.
Thus $\Lambda_5=0$ implies $n_2-n_1\leq 1$ and therefore also $m_2-m_1\leq 1$. In particular we have $n_2-n_1=m_2-m_1=1$ since $n_2>n_1$.

We compute now
\begin{multline*}
|\Lambda_{22}-\Lambda_{21}|=|(m_2-m_1)\log \alpha-(a_2-a_1)\log y|\\
<1.1(\alpha^{-n_1}+\alpha^{-n_1-1})=1.1\alpha^{-n_1+1}\leq 1.1 \alpha^{-2}<0.43
\end{multline*}
since we assume that $n_1>m_1>1$, i.e. $ n_1\geq 3$. Assuming $\Lambda_5=0$ we have $n_1-m_1=1$ and therefore we have 
$$\log y<0.43+\log \alpha<0.92$$
or in other words that $y=2$.

Let us summarize our results so far:

\begin{lemma}\label{lem:Lambda5-vanish}
 We have either $\Lambda_5=0$ and $y=2$ or \eqref{eq:bound-n_1} holds. 
\end{lemma}

\subsection{Final steps}

Next we want to find an upper bound for $\log y$. In view of the previous lemma we may assume that $y>2$ and therefore assume that \eqref{eq:bound-n_1} holds. Further, let us assume for the moment that $n_1-m_1\geq 3$. In this case we consider $\Lambda_{11}$ and obtain
$$|n_1 \log \alpha-a_1 \log y-\log \sqrt 5|<2.03 \alpha^{m_1-n_1}<0.5.$$
Thus we obtain
$$\log y<0.5+\log\sqrt{5}+n_1\log \alpha< 9.29\cdot 10^{27}(\log(3.11 n_2))^6.$$

In the case that $n_1-m_1<3$ we consider $\Lambda_{21}$ and obtain
$$|(n_1-2)\log \alpha-a_1\log y +\log \tau(n_1-m_1)|<0.68,$$
which implies
$$\log y<1.1+\log\tau(3)+n_1\log \alpha<9.29\cdot 10^{27}(\log(3.11 n_2))^6.$$

Thus in any case we conclude that
\begin{equation}\label{eq:ybound}
 \log y< 9.29\cdot 10^{27} (\log(3.11 n_2))^6.
\end{equation}
If we insert this bound into the bound from Proposition \ref{prop:Bravo-Luca-bound}, then we obtain
$$n_2< 2.94\cdot 10^{78} (\log 13.81 n_2)^{14}$$
which yields $n_2<2.03\cdot 10^{112}$. In combination with \eqref{eq:ybound} and \eqref{eq:bound-n_1} we obtain

\begin{proposition}
 Assume that \eqref{eq:main-gen} has two solutions $(n_1,m_1,a_1)$ and $(n_2,m_2,a_2)$, then 
 we have
 $$n_2<2.03\cdot 10^{112} ,\qquad \text{and} \qquad \log y<2.86 \cdot 10^{42} ,\qquad \text{and}\qquad n_1<5.93 \cdot 10^{42}.$$
\end{proposition}

\section{The reduction process}\label{sec:reduction}

Let us assume for technical reasons that $n_1-m_1,n_2-m_2\geq 3$. We consider inequality \eqref{eq:Lambda3} and find that
\begin{equation}\label{eq:Lambda3-red}
\left|(a_2-a_1)\frac{\log\sqrt 5}{\log \alpha}-\Delta \right|<\frac{5.93 n_2}{\min\{\alpha^{n_1-m_1},\alpha^{n_2-m_2}\}}.
\end{equation}
Since $a_2-a_1<0.7 n_2<1.61\cdot 10^{112}$ we compute the first $216$ convergents of $\frac{\log\sqrt 5}{\log \alpha}$ and note that the denominator of the $216$-th continued fraction satisfies 
$$a_2-a_1<0.7 n_2 <q_{216}< 1.97 \cdot 10^{112}.$$
Moreover, we find that
$$A=\max\{a_i\: :\: 1\leq i \leq 216\}=330,$$
which yields due to Lemma \ref{lem:cont-frac}
$$\min\{\alpha^{n_1-m_1},\alpha^{n_2-m_2}\}<5.93 n_2 q_{216} \cdot 332<7.88\cdot 10^{227}.$$
Thus we get
\begin{equation}\label{eq:bound-small-L3}
 \min\{n_1-m_1,n_2-m_2\}\leq 1091.
\end{equation}

We consider now the linear form $\Lambda_{4a}$, that is we consider the case that
$\min\{n_1-m_1,n_2-m_2\}=n_1-m_1$. In view of an application of Lemma \ref{lem:real-reduce} we consider for every $0<t\leq 1091$ the lattice generated by the columns of the matrix
\begin{equation*}
\mathcal{A}=\begin{pmatrix}
    1 & 0  & 0 \\
    0 & 1 & 0 \\
    \lfloor C\log \alpha \rfloor & \lfloor C\log \sqrt 5 \rfloor & \lfloor C\log \tau(t) \rfloor 
\end{pmatrix},
\end{equation*}
with $C$ chosen sufficiently large, in particular will choose $C=10^{467}$. That is we consider the inequality
$$\left|x_1 \log \alpha +x_2 \log \sqrt 5+x_3 \log \tau(t)\right|<c_4 \exp(-c_3 \alpha H),$$
where $x_1=\Delta$, $x_2=a_1$, $x_3=-a_2$, $c_3 =\log \alpha$ and depending whether $\frac{2.2 a_2}{\alpha^{n_1}}$ or $\frac{4.06 a_1}{\alpha^{n_2-m_2}}$ is the maximum in \eqref{eq:Lambda4a-ieq} we choose $c_4=4.06 a_1<1.54 n_1<9.14\cdot 10^{42}$ and $H=n_2-m_2$ or $c_4=4.06 a_2<2.842 n_2<5.77\cdot 10^{112}$ and $H=n_1$.

In view of an application of Lemma \ref{lem:real-reduce} we note that we can choose 
\begin{align*}
X_1&=8.43\cdot 10^{154}>0.7n_2n_1>|\Delta|,\\
X_2&=4.151\cdot 10^{42}>0.7n_1>a_1,\\
X_3&=1.421\cdot 10^{112}>0.7n_2>a_2.
\end{align*}
For all $1\leq t\leq 1091$ with $t\neq 1,2,10$ we get with this choice a constant $c_1$ from Lemma \ref{lem:lattice} such that $c_1^2 \geq T^2+S$ is satisfied. Thus we get in the case that $t=n_1-m_1\neq 1,2,10$ either the upper bound $n_2-m_2\leq 1698$ or $n_1\leq 2031$.

Therefore we are left with the case that $t=1,2$ or $10$. In these cases we have
\begin{align*}
|\Lambda_{4a}|&=|(\Delta-2 a_2)\log \alpha-(a_2-a_1)\log \sqrt{5}|,&
\text{if $t=1$;}\\
|\Lambda_{4a}|&=|(\Delta-a_2)\log \alpha-a_1\log \sqrt{5}|,&
\text{if $t=2$;}\\
|\Lambda_{4a}|&=| (\Delta-5a_1) \log \alpha+(-a_1-2a_2)\log \sqrt{5}|,&
\text{if $t=10$.}
\end{align*}
Note that $|\Lambda_{4a}|<\max\left\{\frac{2.2a_2}{\alpha^{n_1}},\frac{4.06a_1}{\alpha^{n_2-m_2}}\right\}$. That is we are in almost the same situation as in \eqref{eq:Lambda3-red} and we can apply Lemma \ref{lem:cont-frac} similar as before. We omit the details, but let us note that the bounds for $n_2-m_2$ and $n_1$ we obtain by applying Lemma \ref{lem:cont-frac} in the special case that $t=1,2,10$ do not succeed the upper bounds we got from the LLL-reduction method.

Let us consider the linear form $\Lambda_{4b}$ and therefore assume that
$\min\{n_1-m_1,n_2-m_2\}=n_2-m_2$. We consider again the approximation lattice spanned by the columns of
\begin{equation*}
\mathcal{A}=\begin{pmatrix}
    1 & 0  & 0 \\
    0 & 1 & 0 \\
    \lfloor C\log \alpha \rfloor & \lfloor C\log \sqrt 5 \rfloor & \lfloor C\log \tau(t) \rfloor 
\end{pmatrix},
\end{equation*}
with $C=10^{467}$ which is sufficiently large in this case. Similar as before we consider the inequality
$$\left|x_1 \log \alpha +x_2 \log \sqrt 5+x_3 \log \tau(t)\right|<c_4 \exp(-c_3 \alpha H),$$
but this time with $x_1=\Delta$, $x_2=a_2$, $x_3=-a_1$, $c_3 =\log \alpha$, $c_4=5.77\cdot 10^{112}>4.06 a_2$ and $H=n_1-m_1$. Thus we choose 
$X_1=8.43\cdot 10^{154}$, $X_2=1.421\cdot 10^{112}$ and $X_3=4.151\cdot 10^{42}$.
For all $1\leq t\leq 1091$ with $t\neq 1,2,10$ Lemma \ref{lem:real-reduce} yields an upper bound for $n_1-m_1$. In particular, we get $n_1-m_1\leq 2032$. The exceptional cases $t=1,2,10$ can again be dealt with an application of 
Lemma \ref{lem:cont-frac}. Altogether we obtain the following lemma.

\begin{lemma}\label{lem:Red1}
 We have $\min\{n_1-m_1,n_2-m_2\}\leq 1091$ and either $\max\{n_1-m_1,n_2-m_2\}\leq 2032$ or $n_1\leq 2031$. Moreover, we have in any case that $n_2-m_2\leq 1698$.
\end{lemma}

With these bounds inserted into \eqref{eq:Lambda5} we obtain instead of \eqref{eq:bound-n_1} the bound 
$$n_1<1.55 \cdot 10^{17} \log(3.11 n_2)$$
which yields along the same lines provided in Section \ref{sec:bound}
$$\log y<7.46 \cdot 10^{16} \log(3.11 n_2).$$
If we consider now inequality \eqref{eq:Lambda2-ieq} we obtain the inequality
\begin{align*}
n_2 \log \alpha - 0.1 &<-\log |\Lambda_2|\\
&<7.26\cdot 10^{10} (2 \log y) \max\{3,n_2-m_2\} (\log \alpha) \log(13.81 n_2)\\
&<  8.86\cdot 10^{30}  \log(3.11 n_2)\log(13.81 n_2),
\end{align*}
which yields $n_2<1.26\cdot 10^{35}$ and therefore we have
$\log y<6.12 \cdot 10^{18}$ and $n_1<1.28 \cdot 10^{19}$. That is instead of the upper bounds from Proposition \ref{prop:boundn} we may use the upper bounds which are given in the next lemma.

\begin{lemma}\label{lem:Red2}
 We have
 $$n_2<1.26\cdot 10^{35} ,\qquad \text{and} \qquad \log y<6.12 \cdot 10^{18} ,\qquad \text{and}\qquad n_1<1.28 \cdot 10^{19}.$$
\end{lemma}

We use again the continued fraction reduction and consider the first $74$ convergents to find the upper bound
$$\min\{n_1-m_1,n_2-m_2\}\leq 348.$$

If we repeat the LLL-reduction process applied to $\Lambda_{4a}$ and $\Lambda_{4b}$, but this time with the already reduced bounds given in Lemma \ref{lem:Red2} we obtain
$$\max\{n_1-m_1,n_2-m_2\}\leq 698 \qquad \text{or} \qquad n_1\leq 697. $$

We now apply the LLL-reduction to $\Lambda_5$, i.e. we apply it to inequality \eqref{eq:Lambda5}.
That is we consider the lattice generated by the columns of the matrix
\begin{equation*}
\mathcal{A}=\begin{pmatrix}
    1 & 0  & 0 \\
    0 & 1 & 0 \\
    \lfloor C\log \tau(t_1) \rfloor & \lfloor C\log \tau(t_2) \rfloor & \lfloor C\log \alpha \rfloor 
\end{pmatrix},
\end{equation*}
where $1\leq t_1\leq 348$ and $1\leq t_2 \leq 698$ and choose $C=10^{200}$. With this choice we get $a_1,a_2<8.89 \cdot 10^{34}=X_1=X_2$ and 
$|\Delta|<0.7 n_1n_2<1.13\cdot 10^{54}=X_3$. Moreover we take $ c_3=\log \alpha $, $2.2 a_2 < 1.95\cdot 10^{35}=c_4$ and $H=n_1$.

Provided that $t_1,t_2\neq 2$, $t_1\neq t_2$ and $(t_1,t_2)\neq (1,10), (10,1)$ the LLL-reduction method succeeds in showing that $n_1\leq 866$. Let us consider now the exceptional cases.

In the case that $t_1=2$ or $t_2=2$ or $t_1=t_2\neq 2$ we obtain
$$|\Lambda_5|=\left|\tilde \Delta \log \alpha - \tilde a \log \tau(t) \right|$$
with
\begin{itemize}
 \item $\tilde \Delta=\Delta-a_1$ and $\tilde a=-a_2$, if $t_1=2$ and $t=t_2\neq 2$;
 \item $\tilde \Delta=\Delta+a_2$ and $\tilde a=a_1$, if $t_2=2$ and $t=t_1\neq 2$;
 \item $\tilde \Delta=\Delta$ and $\tilde a=a_1-a_2$, if $t_1=t_2=t\neq 2$.
\end{itemize}
We apply Lemma \ref{lem:cont-frac} to the inequality
$$\left| \tilde a \frac{\log \tau(t)}{\log \alpha}-\tilde \Delta \right|<\frac{2.2 a_2}{\log \alpha} \alpha^{-n_1}. $$
Note that in any case $|\tilde a|\leq a_2$ and we compute for all $1\leq t \leq 698$ with $t\neq 2$ the continued fraction of $\mu=\frac{\log\tau(t)}{\log \alpha}$. In any case it is sufficient to compute the first $84$ convergents to find a convergent $p_\ell/q_\ell$ such that $q_\ell>a_2$. In particular, we find in any case an upper bound for $n_1$ and altogether obtain $n_1<362 $ in this case.

Similarly in the case that $(t_1,t_2)=(1,10),(10,1)$ we get
$$|\Lambda_5|=\left|\tilde \Delta \log \alpha - \tilde a \log \sqrt{5} \right|$$
with
\begin{itemize}
 \item $\tilde \Delta=\Delta-5a_1+2a_2$ and $\tilde a=2a_1+a_2$, if $t_1=1$ and $t_2=10$;
 \item $\tilde \Delta=\Delta+5a_2-2a_1$ and $\tilde a=-2a_2-a_1$, if $t_1=10$ and $t_2=1$.
\end{itemize}
In both cases we obtain 
$$\left|\tilde a \frac{\log \sqrt{5}}{\log \alpha}-\tilde \Delta \right|<\frac{2.2 a_2}{\log \alpha} \alpha^{-n_1}$$
by an application of Lemma \ref{lem:cont-frac} to this inequality we obtain $n_1\leq 345$.

In the case that $t_1=t_2=2$ and $\Lambda_5\neq 0$ we get
$$\log \alpha\leq |\tilde \Delta \log \alpha|=|\Lambda_5|<2.2\frac{a_2}{\alpha^n_1},$$
which implies that $n_1\leq 168$. Note that in the case that $\Lambda_5=0$ we have by Lemma~\ref{lem:Lambda5-vanish} that $y=2$.

Let us assume for the moment that $y\neq 2$. Then we get that $n_1\leq 866$, therefore we conclude that $\log y<418.3$. Considering inequality \eqref{eq:Lambda2-ieq} yields then $n_2<3.86\cdot 10^{18}$.

In the case that $y=2$ we can directly consider inequality \eqref{eq:Lambda2-ieq} and get $n_2<2.69\cdot 10^{15}$.

With this new bound for $n_2$ we can do the whole reduction process once again. And obtain first
that $\min\{n_1-m_1,n_2-m_2\}\leq 188$ by using Lemma \ref{lem:cont-frac}. Then we obtain by the LLL-reduction that $\max\{n_1-m_1,n_2-m_2\}\leq 305$ and again using LLL-reduction we finally obtain that $n_1\leq 470$. However, this implies $\log y< 227.71$ and $n_2< 4.37\cdot 10^{17}$.

Thus we have proved so far:

\begin{proposition}\label{prop:boundn}
 Assume that for a fixed $y>1$ there exist two solutions $(n_1,m_1,a_1)$ and $(n_2,m_2,a_2)$ to \eqref{eq:main-gen} with $n_1<n_2$. Then we have $\log y< 227.71$  and  $n_2< 4.37\cdot 10^{17}$. Moreover, if $y\neq 2$, then we also have that $n_1\leq 470$.
\end{proposition}

That is the possible list of $y$'s such that \eqref{eq:main-gen} has more than one solution can be characterized as follows:

\begin{corollary}\label{cor:y-char}
 Let $y\neq 2$ be fixed and assume that the Diophantine equation \eqref{eq:main-gen} has more than one solution. Then there exist integers $1<m<n\leq 470$ and a positive integer $a$ such that $F_n+F_m=y^a$.
\end{corollary}

\section{The final steps}

According to Corollary \ref{cor:y-char} we are left to solve \eqref{eq:main-gen} for every $y$ such that 
\begin{itemize}
\item $y^a=F_n+F_m$ for some $a>0$ and $1<m<n\leq 470$ or
\item $y=2$. 
\end{itemize}
Therefore we are left to solve about $\binom{470}2 \sim 110000$ Diophantine equations.

Let us assume that $y\neq 2$. To solve all these equations we proceed as follows. For every pair of integers $(n_1,m_1)$, with $1<m_1<n_1 \leq 470$ we perform the following steps which have been implemented in Sage \cite{sagemath}:

\textbf{Step I:} We compute $\tilde y=F_{n_1}+F_{m_1}$ and find the largest exponent $a_1$ such that there exists an integer $y$ such that $y^{a_1}=\tilde y$. In most cases we will have $a_1=1$ and $y=\tilde y$.

\textbf{Step II:} We compute $\mu=\frac{\log \alpha}{\log y}$ to a sufficiently high precision (in our case $3000$-bit precision was sufficiently large) and compute then the continued fraction expansion of $\mu$. More precisely we compute successively the convergents $p_\ell/q_\ell$ to $\mu$. For each convergent we compute
$$\kappa=\frac{1}{\|q_\ell \mu\|2N}$$
where $N=4.37\cdot 10^{17}$ is the upper bound for $n_2$. If $\kappa>1$ we store the pair $(\kappa,q_\ell)$ in the list \texttt{KAPPA}. If we have found $250$ such pairs we stop\footnote{The reason why we have to compute so many pairs will be explained in Remark \ref{rem:pairs}}.

\textbf{Step III:} We apply the Baker-Davenport reduction (Lemma \ref{lem:BakDav1}) to the inequality
$$\left|n_2\frac{\log \alpha}{\log y}+\frac{\log \sqrt{5}}{\log y}-a\right|<\frac{2.03}{(\log y) \alpha^{n_2-m_2}} $$
which is easily obtained from inequality \eqref{eq:Lambda1}. For all pairs $(\kappa,q)$ in the List \texttt{KAPPA} we compute $\left\|q \tau\right\|$, with $\tau=\frac{\log \sqrt 5}{\log y}$ until we find a pair $(q,\kappa)$ such that $\|q\tau\|>\frac 1{\kappa}$. Then according to Lemma \ref{lem:BakDav1} we have 
$$n_2-m_2<\frac{\log(2\kappa q 2.03/\log y)}{\log \alpha}=\tilde N.$$

Let us note that for all instances we can find such a pair $(\kappa,q)$ and are able to compute a rather small bound $n_2-m_2$ in each case.

\textbf{Step IV:} For each $1\leq t \leq \tilde N$ we consider the inequality
$$\left|n_2\frac{\log \alpha}{\log y}+\frac{\log \tau(t)}{\log y}-a\right|<\frac{1.1}{(\log y) \alpha^{n_2}} $$
which we deduce from inequality \eqref{eq:Lambda2}. In most cases we can find a pair $(\kappa,q)$ that satisfies $\|q \tau\|>1/\kappa$ with $\tau= \frac{\log \tau(t)}{\log y}$ and therefore we find in the most cases a small upper bound for $n_2$.

The exceptional cases, where we could not find a suitable pair $(\kappa,q)$ are the following:
\begin{itemize}
 \item $(n_1,m_1,t)=(4,2,6),(5,4,6),(7,4,6)$, i.e. $y=2$ and $t=6$;
 \item $(n_1,m_1,t)=(n,n-1,2n+2)$ with $n$ is even.
\end{itemize}
These special cases will be treated in Steps V and VI.

\textbf{Step V:} We consider the case that $y=2$ and $t=6$. Let us note that this case was also a special case in the paper of Bravo and Luca \cite{Bravo:2016}. They used a special identity involving Lucas numbers to resolve this case. To keep this paper as self contained as possible we resolve this case by using Lemma \ref{lem:cont-frac}.

Therefore we note that $\tau(6)=\frac{\alpha^6+1}{\sqrt{5}}=2 \alpha^3$ and inequality \eqref{eq:Lambda2} turns into
$$|(n_2+3)\log \alpha - (a_2-1)\log 2|<\frac{1.1}{\alpha^{-n_2}}.$$
An application of Lemma \ref{lem:cont-frac} yields again a small upper bound for $n_2$. In particular we obtain $n_2 \leq 94$ in this case.

\textbf{Step VI:} Now, we consider the case that $n_1=m_1+1$, with $n_1$ even and $t=n_2-m_2=2(n_1+1)$. Let us compute
\begin{align*}
 \tau(t)&=\frac{\alpha^t+1}{\sqrt{5}}=\frac{\alpha^{2n_1+2}+1}{\sqrt{5}}\\
 &=\alpha^{n_1+1} \frac{\alpha^{n_1+1}+\alpha^{-n_1-1}}{\sqrt{5}}\\
 &=\alpha^{n_1+1} \frac{\alpha^{n_1+1}-\beta^{n_1+1}}{\sqrt{5}}=\alpha^{n_1+1} F_{n_1+1}\\
 &=\alpha^{n_1+1}(F_{n_1}+F_{m_1})=\alpha^{n_1+1}y^{a_1}.
\end{align*}
Therefore inequality \eqref{eq:Lambda2} turns into
$$|(n_2+n_1+1)\log \alpha -(a_2-a_1)\log y|<\frac{1.1}{\alpha^{-n_2}} $$
and we can apply Lemma \ref{lem:cont-frac} to deduce a small upper bound for $n_2$ in each case.

Therefore we have found for each instance a rather small upper bound $\tilde N_2$ for $n_2$. Due to Lemma \ref{lem:el-bounds} we also have 
$$a_2<\frac{\tilde N_2 \log \alpha}{\log y},$$
which provides a small upper bound $A$ for $a_2$.

\textbf{Step VII:} For each $1\leq a \leq A$ we compute the Zeckendorf expansion of $y^a$ by a greedy digit expansion algorithm. If the Zeckendorf expansion of $y^a$ has more than two non-zero digits, then there exists no solution $(n,m,a)$ to \eqref{eq:main-gen} and we can discard $a$. If the Zeckendorf expansion is of the form
$y^a=F_n+F_m$ with $n-m>1$ we have found a solution to \eqref{eq:main-gen}, namely $(n,m,a)$. In the
case that the Zeckendorf expansion is of the form $y^a=F_n$, then $(n-1,n-2,a)$ is a solution to \eqref{eq:main-gen}. Thus we find all possible solutions to \eqref{eq:main-gen}.

In all cases except, those cases which correspond to $y=2,3,4,6,10$ we have found exactly one solution, of course the initial solution $F_{n_1}+F_{m_1}=y^a=\tilde y$. Therefore our main theorem, Theorem \ref{th:main} is proved in the case that $y\neq 2$.

Finally we have to discuss the case that $y=2$. Of course this case has already been resolved by Bravo and Luca \cite{Bravo:2016}. However, also our computations made so far are sufficient to resolve this case. Therefore we note that $4=2^2=F_4+F_2$, thus the case $y=2$ has been covered by our previous computations. In particular, the solutions we get by considering the pair $(n,m)=(4,2)$ provide all solutions in the case that $y=2$. Thus the proof of Theorem \ref{th:main} is complete.

Let us note that the computation of all steps for all instances $2\leq m_1<n_1\leq 470$ took about four hours on a common desk top PC (Intel Core i7-8700) on a single core. 

\begin{remark}\label{rem:pairs}
 Applying the Baker-Davenport reduction to Inequality \eqref{IEq:BakerDav} with randomly chosen $\mu$ and $\tau$ the list \texttt{KAPPA} computed in Step II can be rather short (in many cases the first five $\kappa$'s will be sufficient). But, in our case the choice of $\mu$ and $\tau$ corresponds to a large solution of our main equation \eqref{eq:main-gen}, namely to the solution $F_{n_1}+F_{m_1}=y^{a_1}$. These rather large solutions yield good simultaneous approximations of our quantities $\mu$ and $\tau$, thus we need a rather large $\kappa$ to be successful in applying the Baker-Davenport reduction method.
 
 For instance, let us consider the case that $n_1=m_1+1$, $t=2(n_1+1)$ and $n_1$ is odd. Then we obtain by a similar computation as in Step VI that
 $$\tau(t)=\alpha^{n_1+1}\frac{\alpha^{n_1+1}+\beta^{n_1+1}}{\sqrt 5}=\alpha^{n_1+1}y^{a_1}+\frac{2}{\sqrt 5}.$$
 and therefore we have
 $$\tau=\frac{\log \tau(t)}{\log y}=\frac{(n_1+1)\log \alpha+a_1\log y +r}{\log y}=(n_1+1)\mu+a_1+\frac{r}{\log y},$$
 with $|r|<4\alpha^{-2n_1-2}$ and $\mu=\frac{\log \alpha}{\log y}$. Therefore, only a pair $(\kappa,q)$ with large $\kappa$ can satisfy the necessary assumptions that $\|q\mu\|<\frac{1}{2\kappa N}$ and $\|q\tau\|>\frac{1}{\kappa}$ both hold. Thus we have to compute an unusual long list \texttt{Kappa} in Step II, to cover the case that $n_1=m_1+1$, $t=2(n_1+1)$ and $n_1$ is odd. 
\end{remark}


\begin{thebibliography}{10}

\bibitem{Baker:NT}
A.~Baker.
\newblock {\em A concise introduction to the theory of numbers}.
\newblock Cambridge University Press, Cambridge, 1984.

\bibitem{Baker:1969}
A.~Baker and H.~Davenport.
\newblock The equations $3x^2-2=y^2$ and $8x^2-7=z^2$.
\newblock {\em Quart. J. Math. Oxford Ser. (2)}, 20:129--137, 1969.

\bibitem{Bravo:2014}
J.~J. Bravo and F.~Luca.
\newblock Powers of two as sums of two {L}ucas numbers.
\newblock {\em J. Integer Seq.}, 17(8):Article 14.8.3, 12, 2014.

\bibitem{Bravo:2016}
J.~J. {Bravo} and F.~{Luca}.
\newblock {On the Diophantine equation \(F_n + F_m=2^a\)}.
\newblock {\em {Quaest. Math.}}, 39(3):391--400, 2016.

\bibitem{deWeger:1987}
B.~M.~M. de~Weger.
\newblock Solving exponential {D}iophantine equations using lattice basis
  reduction algorithms.
\newblock {\em J. Number Theory}, 26(3):325--367, 1987.

\bibitem{Kebli:2021}
S.~Kebli, O.~Kihel, J.~Larone, and F.~Luca.
\newblock On the nonnegative integer solutions to the equation 
$f_n \pm f_m=y^a$.
\newblock {\em Journal of Number Theory}, 220:107 -- 127, 2021.

\bibitem{Laurent:2008}
M.~{Laurent}.
\newblock {Linear forms in two logarithms and interpolation determinants. II}.
\newblock {\em {Acta Arith.}}, 133(4):325--348, 2008.

\bibitem{Luca:2018}
F.~{Luca} and V.~{Patel}.
\newblock {On perfect powers that are sums of two Fibonacci numbers}.
\newblock {\em {J. Number Theory}}, 189:90--96, 2018.

\bibitem{Matveev:2000}
E.~M. Matveev.
\newblock An explicit lower bound for a homogeneous rational linear form in
  logarithms of algebraic numbers. {II}.
\newblock {\em Izv. Ross. Akad. Nauk Ser. Mat.}, 64(6):125--180, 2000.

\bibitem{Schinzel:1974}
A.~{Schinzel}.
\newblock {Primitive divisors of the expression A\(^n-B^n\) in algebraic number
  fields}.
\newblock {\em {J. Reine Angew. Math.}}, 268/269:27--33, 1974.

\bibitem{Schinzel:1993}
A.~{Schinzel}.
\newblock {An extension of the theorem on primitive divisors in algebraic
  number fields}.
\newblock {\em {Math. Comput.}}, 61(203):441--444, 1993.

\bibitem{Smart:DiGL}
N.~P. Smart.
\newblock {\em The algorithmic resolution of {D}iophantine equations},
  volume~41 of {\em London Mathematical Society Student Texts}.
\newblock Cambridge University Press, Cambridge, 1998.

\bibitem{sagemath}
{The Sage Developers}.
\newblock {\em {S}ageMath, the {S}age {M}athematics {S}oftware {S}ystem
  ({V}ersion 9.0)}, 2020.
\newblock {\tt https://www.sagemath.org}.

\end{thebibliography}

\def\cprime{$'$}

\end{document}